\documentclass[12pt,leqno,a4paper]{article}

\usepackage{amsmath}
\usepackage{amssymb}
\usepackage{amsthm}
\usepackage[UKenglish]{babel}
\usepackage[pdftitle={Quantum random walks and thermalisation},
            pdfauthor={Alexander C. R. Belton},
            a4paper,
            bookmarks=true,
            bookmarksnumbered=true,
            colorlinks=false,
            dvips=true,
            pdfstartview=FitH]{hyperref}
\usepackage[left=3.15cm,right=3.15cm,top=3.5cm,bottom=5.4cm]{geometry}
\usepackage{ifthen}
\usepackage{titlesec}

\titleformat{\section}%
{\normalfont\normalsize\bfseries}{\thesection}{1em}{}
\titlespacing{\section}%
{0pt}{1.25ex plus 0.5ex minus 0.25ex}{1ex plus 0.1ex}
\titleformat{\subsection}%
{\normalfont\normalsize\bfseries}{\thesubsection}{1em}{}
\newlength{\oursubsecgap}
\titlespacing{\subsection}%
{0pt}{0.75ex plus 0.2ex minus 0.1ex}{\oursubsecgap}

\mathchardef\ordinarycolon\mathcode`\:
\mathcode`\:=\string"8000
\def\vcentcolon{\mathrel{\mathop\ordinarycolon}}
\begingroup \catcode`\:=\active
  \lowercase{\endgroup
  \let :\vcentcolon
  }

\theoremstyle{plain}
\newtheorem{thm}{Theorem}[section]

\newtheorem{lem}[thm]{Lemma}
\newtheorem{prop}[thm]{Proposition}

\theoremstyle{definition}
\newtheorem{defn}[thm]{Definition}
\newtheorem{exmp}[thm]{Example}
\newtheorem{notn}[thm]{Notation}
\newtheorem{rem}[thm]{Remark}

\newcommand{\bebe}{\Gamma}
\newcommand{\cga}{\phi}
\newcommand{\cgb}{\psi}
\newcommand{\diag}{\delta}
\newcommand{\dialg}{\mathcal{D}}
\newcommand{\elltwo}{L^2( \R_+; \mul )}
\newcommand{\evec}[1]{\evecc( #1 )}
\newcommand{\evecc}{\varepsilon}
\newcommand{\evecs}{\mathcal{E}}
\newcommand{\fock}{\mathcal{F}}
\newcommand{\half}{\mbox{$\frac{1}{2}$}}
\newcommand{\ham}{H}
\newcommand{\hdiag}{\ham_\rd}
\newcommand{\hilb}{\mathsf{h}}
\newcommand{\hilc}{\mathsf{H}}
\newcommand{\hint}{\ham_{\mathrm{int}}}
\newcommand{\hoffd}{\ham_{\mathrm{o}}}
\newcommand{\hpar}{\ham_{\mathrm{par}}}
\newcommand{\hsys}{\ham_{\mathrm{sys}}}
\newcommand{\htot}{\ham_{\mathrm{tot}}}
\newcommand{\id}{I}
\newcommand{\indf}[1]{\indff_{#1}}
\newcommand{\indff}{1}
\newcommand{\ini}{\mathsf{h}}
\newcommand{\opsp}{\mathsf{V}}
\newcommand{\opsq}{\mathsf{W}}
\newcommand{\mmul}{{\widehat{\mul}}}
\newcommand{\mmulb}{{{\widehat{\mul}}\rb}}
\newcommand{\mul}{\mathsf{k}}
\newcommand{\ppi}{\widetilde{\pi}}
\newcommand{\sstate}{\widetilde{\state}}
\newcommand{\State}{\varrho}
\newcommand{\state}{\rho}
\newcommand{\statesp}{\mathsf{K}}
\newcommand{\tto}[1]{\rightarrow_{#1}}
\newcommand{\vac}{\omega}
\newcommand{\wga}{\Phi}
\newcommand{\wgb}{\Psi}

\newcommand{\C}{\mathbb{C}}
\newcommand{\R}{\mathbb{R}}
\newcommand{\Z}{\mathbb{Z}}

\newcommand{\algten}{\odot}
\newcommand{\bop}[2]%
{\ifthenelse{\equal{#2}{}}{\bopp( #1 )}{\bopp( #1; #2 )}}
\newcommand{\bopp}{B}
\newcommand{\cb}{\mathrm{cb}}
\newcommand{\dyad}[2]{| #1 \rangle \langle #2 |}

\newcommand{\etc}{\textit{et cetera}}
\newcommand{\hb}{{\hilc\rb}}
\newcommand{\I}{\mathrm{i}}
\newcommand{\ie}{\textit{i.e., }}
\newcommand{\kbo}[3]
{\ifthenelse{\equal{#3}{}}{#1 \kboo( #2 )}{#1 \kboo( #2; #3 )}}
\newcommand{\kboo}{\bopp}
\newcommand{\lift}[2]{{#1} \matten {\id_{\bop{#2}{}}}}
\newcommand{\mat}[2]{{#2} \matten \bopp( #1 )}
\newcommand{\matt}{\mathrm{M}}
\newcommand{\matten}{\mathop{\otimes_\matt}}
\newcommand{\mmodf}[2]{m( #1, #2)}
\newcommand{\nmodf}[2]{m_d( #1, #2)}
\newcommand{\rb}{\mathrm{b}}
\newcommand{\rd}{\mathrm{d}}
\newcommand{\re}{\mathrm{e}}
\newcommand{\rf}{\mathrm{f}}
\newcommand{\std}{\,\rd}
\newcommand{\tr}{\mathop{\mathrm{tr}}\nolimits}
\newcommand{\uwkten}{\mathop{\overline{\otimes}}}
\newcommand{\vna}{\mathcal{M}}

\renewcommand{\geq}{\geqslant}
\renewcommand{\leq}{\leqslant}

\begin{document}

\begin{center}
{\Large Quantum random walks and thermalisation}\\[1ex]
{\normalsize Alexander C. R. Belton}\\[1ex]
{\small Department of Mathematics and Statistics\\
Lancaster University, United Kingdom\\[0.5ex]
\href{mailto:a.belton@lancaster.ac.uk}%
{\textsf{a.belton@lancaster.ac.uk}} \qquad \today}
\end{center}

\begin{abstract}\noindent
It is shown how to construct quantum random walks with particles in an
arbitrary faithful normal state. A convergence theorem is obtained for
such walks, which demonstrates a thermalisation effect: the limit
cocycle obeys a quantum stochastic differential equation without gauge
terms. Examples are presented which generalise that of Attal and Joye
(2007, \textit{J.~Funct.\ Anal.}~\textbf{247}, 253--288).
\end{abstract}

{\footnotesize\textit{Key words:} quantum random walk; thermalisation;
thermalization; quantum heat bath; repeated interactions; quantum
Langevin equation; toy Fock space.

}

\vspace{1ex}
{\footnotesize\textit{MSC 2000:} %
81S25 (primary);    
46L07,              
46L53,              
46N50 (secondary).  

}

\section{Introduction}

A quantum random walk \cite{Blt08} may be interpreted as the dynamics
of a quantum system that interacts periodically with a stream of
identical particles, each of which lies in the vacuum state. As
observed in \cite{AtJ07}, a Gelfand--Naimark--Segal construction may
be employed to consider particles in a more general state (see also
Franz and Skalski~\cite{FrS07}); for a particular choice of evolution,
one generated by a Hamiltonian which describes dipole-type
interaction, Attal and Joye demonstrated that the limit flow involves
fewer quantum noises than might na\"{i}vely be expected, a so-called
`thermalisation' effect \cite[Theorem~7]{AtJ07}.

Inspired by Attal and Joye's work, techniques from \cite{Blt08} are
used below to show that thermalisation occurs for a large class of
quantum random walks. Let $\state$ be a faithful normal state on the
particle algebra $\bop{\statesp}{}$, where the Hilbert space
$\statesp$ may be infinite dimensional, and let $d$ be a
$\state$-preserving conditional expectation on $\bop{\statesp}{}$.
Suppose the linear maps
$\wga(\tau) : \bop{\ini}{} \to \bop{\ini \otimes \statesp}{}$ are such
that, in a suitable sense,
\[
\tau^{-1} \diag( \wga(\tau)( a ) - a \otimes \id_\statesp ) + %
\tau^{-1 / 2} \diag^\perp( \wga(\tau)( a ) ) \to \wgb( a ) \qquad %
\mbox{as } \tau \to {0+}
\]
for all $a \in \bop{\ini}{}$, where
$\diag := \id_{\bop{\ini}{}} \uwkten d$ and
$\diag^\perp := \id_{\bop{\ini \otimes \statesp}{}} - \diag$. Then the
random walk with generator $\wga(\tau)$ and particle state $\state$
converges to a quantum stochastic cocycle~$k$ on $\ini \otimes \fock$,
where $\fock$ is the Boson Fock space over $\elltwo$, which satisfies
the following quantum Langevin equation:
\begin{equation}\label{eqn:coqsde}
k_t( a ) = a \otimes \id_\fock + %
\sum_{\alpha, \beta} \int_0^t k_s( \cgb^\alpha_\beta( a ) ) %
\std\Lambda^\beta_\alpha( s ) %
\qquad \forall\, t \in \R_+, \ a \in \bop{\ini}{}.
\end{equation}
The generator of this cocycle is a linear mapping
$\cgb : \bop{\ini}{} \to \bop{\ini\otimes\mmul}{}$, where~$\mmul$ is
the Hilbert space in the GNS representation corresponding to $\state$,
the vector $\vac$ gives the associated state and
$\mul := \mmul \ominus \C\vac$. The generator $\psi$, given explicitly
in terms of~$d$ and $\wgb$, is such that $\cgb^\alpha_\beta = 0$
unless $\alpha = 0$ or $\beta = 0$: there is no contribution from the
gauge integrals in~(\ref{eqn:coqsde}). If $n := \dim \statesp$ is
finite then there are $n^2 - 1$ independent quantum noises when
working with particles in the vacuum state and at most twice that in
the situation described above, not $n^4 - 1$ as might first appear
necessary: this is the thermalisation phenomenon.

The formulation adopted below uses matrix spaces over operator spaces,
the Lindsay--Wills approach to quantum stochastics. This setting is
both natural and fruitful, allowing consideration of walks on
$C^*$~algebras \cite{Blt08} and quantum groups \cite{LiS08}, for
example.

This note is structured as follows: Section~\ref{sec:prelim} rapidly
introduces quantum random walks and presents the main convergence
theorem from~\cite{Blt08}; Section~\ref{sec:results} contains the
result outlined above, Theorem~\ref{thm:main}, and its proof, which is
remarkably simple; Section~\ref{sec:egs} presents two classes of
examples, which include the Attal--Joye model (Examples~\ref{eg:ajhp}
and~\ref{eg:ajeh}).

\subsection{Conventions and notation}

All vector spaces have complex scalar field; all inner products are
linear in the second argument. An empty sum or product equals the
appropriate additive or multiplicative unit, respectively.

The symbol $:=$ is to be read as `is defined to equal' (or
similarly). The indicator function of a set~$A$ is denoted by
$\indf{A}$. The sets of non-negative integers and non-negative real
numbers are denoted by $\Z_+ := \{ 0, 1, 2, \ldots \}$ and
$\R_+ := {[ 0, \infty [}$. The identity operator on a vector space $V$
is denoted by~$\id_V$. Algebraic, Hilbert-space and ultraweak tensor
products are denoted by $\algten$, $\otimes$ and $\uwkten$,
respectively, with the symbol $\otimes$ also denoting the spatial
tensor product of operator spaces. The Dirac dyad $\dyad{u}{v}$
denotes the linear transformation on an inner-product space $V$ such
that $w \mapsto \langle v, w \rangle_V u$.

\section{Preliminaries}\label{sec:prelim}

\subsection{Quantum random walks}

\begin{defn}
Let $\opsp$ be a concrete operator space acting on the Hilbert
space~$\hilb$, \ie a closed subspace of $\bop{\hilb}{}$. Recall that
the \emph{matrix space}
\[
\mat{\hilc}{\opsp} := %
\{ T \in \bop{\hilb \otimes \hilc}{} : %
E^x T E_y \in \opsp \ \ \forall\, x, y \in \hilc \}
\]
is an operator space for any Hilbert space $\hilc$, where
$E^x \in \bop{\hilb \otimes \hilc}{\hilb}$ is the adjoint of
$E_x : u \mapsto u \otimes x$.

The inclusions
$\opsp \otimes \bop{\hilc}{} \subseteq \mat{\hilc}{\opsp} %
\subseteq \opsp \uwkten \bop{\hilc}{}$
hold, with the latter an equality if and only if $\opsp$ is
ultraweakly closed, and there is the natural identification
$\mat{\hilc_2}{( \mat{\hilc_1}{\opsp} ) } = %
\mat{\hilc_1 \otimes \hilc_2}{\opsp}$.
\end{defn}

\begin{defn}
Suppose $\hilc\neq\{0\}$ and let $\opsq$ be another operator space. A
linear map $\wga : \opsp \to \opsq$ is \emph{$\hilc$ bounded} if it is
completely bounded, or it is bounded and the space $\hilc$ is finite
dimensional. The Banach space of such $\hilc$-bounded maps is denoted
by $\kbo{\hilc}{\opsp}{\opsq}$ and is equipped with the
\emph{$\hb$ norm}
\[
\| \cdot \|_\hb : \wga \mapsto \| \wga \|_\hb :=%
\left\{\begin{array}{ll}
(\dim \hilc) \, \|\wga\| & \mbox{if } \dim \hilc < \infty, \\[0.5ex]
\| \wga \|_\cb & \mbox{if } \dim \hilc = \infty,
\end{array} \right.
\]
where $\| \cdot \|_\cb$ denotes the completely bounded norm.
\end{defn}

\begin{rem}
Note that $T \mapsto E^x T E_y$ equals the slice map
$\id_{\bop{\hilb}{}} \uwkten \omega_{x,y}$, where $\omega_{x,y}$ is
the ultraweakly continuous functional
$X \mapsto \langle x, X y \rangle_\hilc$. As
$\{ \omega_{x,y}:x,y\in\hilc\}$ is total in the predual
$\bop{\hilc}{}_*$, it follows that
\[
(\id_{\bop{\hilb}{}} \uwkten \omega)( T ) \in \opsp %
\qquad \forall\, \omega \in \bop{\hilc}{}_*, \ %
T \in \opsp \matten \bop{\hilc}{}.
\]

If $m : \bop{\hilc_1}{} \to \bop{\hilc_2}{}$ is linear,
$\hilb$~bounded and ultraweakly continuous then so is
\[
\id_{\bop{\hilb}{}} \uwkten m : %
\bop{\hilb \otimes \hilc_1}{} \to \bop{\hilb \otimes \hilc_2}{},
\]
by \cite[Lemma~1.5(b)]{dCH85}, and the previous observation shows that
such an ampliation respects the matrix-space structure:
$( \id_{\bop{\hilb}{}} \uwkten m )( \mat{\hilc_1}{\opsp} ) %
\subseteq \mat{\hilc_2}{\opsp}$.
\end{rem}

The following type of ampliation was introduced by Lindsay and Wills
\cite{LiW01}.

\begin{prop}\label{prp:lift}
If $\wga \in \kbo{\hilc}{\opsp}{\opsq}$ then the
\emph{$\hilc$ lifting of~$\wga$} is the unique map
$\lift{\wga}{\hilc} : \mat{\hilc}{\opsp} \to \mat{\hilc}{\opsq}$
such that
\[
E^x ( \lift{\wga}{\hilc}( T ) ) E_y = \wga( E^x T E_y )%
\qquad\forall\, x, y \in \hilc, \ T \in \mat{\hilc}{\opsp}.
\]
The lifting is linear, $\hilc$ bounded and is completely bounded if
$\wga$ is. It satisfies the inequalities
$\| \lift{\wga}{\hilc} \| \leq \| \wga \|_\hb$ and
$\| \lift{\wga}{\hilc} \|_\cb \leq \| \wga \|_\cb$.
\end{prop}
\begin{proof}
See \cite[Theorem~2.5]{Blt08}.
\end{proof}

\begin{prop}\label{prp:qrw}
For any $\wga \in \kbo{\hilc}{\opsp}{\mat{\hilc}{\opsp}}$ there exists
a unique family of maps
$\wga^{(n)} : \opsp \to \mat{\hilc^{\otimes n}}{\opsp}$ indexed by
$n \in \Z_+$, the \emph{quantum random walk} with
\emph{generator}~$\wga$, such that $\wga^{(0)} = \id_\opsp$ and
\[
E^x \wga^{(n + 1)}( a ) E_y = \wga^{(n)}( E^x \wga( a ) E_y )\qquad%
\forall\, x, y \in \hilc, \ a \in \opsp, \ n \in \Z_+.
\]
These maps are linear, $\hilc$ bounded and completely bounded if
$\wga$ is; if $n \geq 1$ then
$\| \wga^{(n)} \|_\hb \leq \| \wga \|_\hb^n$ and
$\| \wga^{(n)} \|_\cb \leq\| \wga \|_\cb^n$.
\end{prop}
\begin{proof}
Given $\wga^{(n)}$, use Proposition~\ref{prp:lift} to let
$\wga^{(n + 1)} := %
( \wga^{(n)} \matten \id_{\bop{\hilc}{}} ) \circ \wga$.
For the first inequality, see \cite[Theorem~2.7]{Blt08}; the second
is immediate.
\end{proof}

\subsection{Toy and Boson Fock space}

\begin{notn}
Let $\mmul$ be a Hilbert space containing the distinguished unit
vector~$\vac$ and let $\mul := \mmul \ominus \C \vac$, the orthogonal
complement of $\C \vac$ in $\mmul$. Define
$\widehat{x} := \vac + x \in \mmul$ for any $x \in \mul$.
\end{notn}

\begin{defn}
The \emph{toy Fock space over $\mul$} is
$\bebe := \bigotimes_{n = 0}^\infty \mmul_{(n)}$, where
$\mmul_{(n)} := \mmul$ for all $n \in \Z_+$, with respect to the
stabilising sequence $(\vac_{(n)} := \vac)_{n = 0}^\infty$; the suffix
$(n)$ is used to indicate the relevant copy of $\mmul$. Note that
$\bebe = \bebe_{n[} \otimes \bebe_{[n}$, where
$\bebe_{n[} := \bigotimes_{m = 0}^{n - 1}\mmul_{(m)}$ and
$\bebe_{[n} := \bigotimes_{m = n}^\infty\mmul_{(m)}$, for all
$n \in \Z_+$.
\end{defn}

\begin{defn}
Let $\fock$ be the Boson Fock space over $\elltwo$, the Hilbert space
of square-integrable $\mul$-valued functions on the half line. Recall
that~$\fock$ may be considered as the completion of $\evecs$, the
linear span of \emph{exponential vectors}~$\evec{f}$ labelled by
$f \in \elltwo$, with respect to the inner product
\[
\langle \evec{f}, \evec{g} \rangle_\fock := \exp\Bigl( %
\int_0^\infty \langle f( t ), g( t ) \rangle_\mul \std t %
\Bigr) \qquad\forall\, f, g \in \elltwo.
\]
\end{defn}

The following gives sense to the idea that the toy space $\bebe$
approximates $\fock$.

\begin{prop}
For all $\tau > 0$ there is a unique co-isometry
$D_\tau : \fock \to \bebe$ such that
\[
D_\tau \evec{f} = \bigotimes_{n = 0}^\infty \widehat{f( n; \tau )}, %
\quad \mbox{where} \quad %
f(n; \tau) := \tau^{-1 / 2} \int_{n \tau}^{(n + 1) \tau}f( t )\std t,
\]
for all $f \in \elltwo$. Furthermore, $D_\tau^* D_\tau \to \id_\fock$
strongly as $\tau \to {0+}$. (See \cite{Blt07}.)
\end{prop}

\subsection{QS cocycles}

\begin{defn}
An \emph{$\ini$ process} $X$ is a family $\{ X_t \}_{t \in \R_+}$ of
linear operators in $\ini \otimes \fock$, such that the domain of each
operator contains $\ini \algten \evecs$ and the map
$t \mapsto X_t u \evec{f}$ is weakly measurable for all $u \in \ini$
and $f \in \elltwo$; this process is \emph{adapted} if
\[
\langle u \evec{f}, X_t v \evec{g} \rangle_{\ini \otimes \fock} = %
\langle u \evec{\indf{[ 0, t [} f}, %
X_t v \evec{\indf{[ 0, t [} g} \rangle_{\ini \otimes \fock} %
\langle \evec{\indf{[ t, \infty [} f}, %
\evec{\indf{[ t, \infty [} g} \rangle_\fock
\]
for all $u$,~$v \in \ini$, $f$,~$g \in \elltwo$ and $t \in \R_+$. (As
is conventional, the tensor-product sign is omitted between elements
of $\ini$ and exponential vectors.)

A \emph{mapping process} $k$ is a family
$\{ k_\cdot( a ) \}_{a \in \opsp}$ of $\ini$ processes such that the
map $a \mapsto k_t( a )$ is linear for all $t \in \R_+$; this process
is \emph{adapted} if each $k_\cdot( a )$ is, it is
\emph{strongly regular} if
\[
k_t( \cdot ) E_{\evec{f}} \in %
\bopp( \opsp; \bop{\ini}{\ini \otimes \fock})%
\qquad \forall\, f \in \elltwo,\ t \in \R_+,
\]
with norm locally uniformly bounded as a function of $t$, and it is
\emph{CB regular} if these two conditions hold with ``bounded''
replaced by ``completely bounded''.
\end{defn}

\begin{thm}
For any $\cga \in \kbo{\mmul}{\opsp}{\mat{\mmul}{\opsp}}$ there exists
a unique strongly regular adapted mapping process $k^\cga$, the
\emph{QS cocycle generated by~$\cga$}, such that
\[
\langle u \evec{f}, ( k^\cga_t( a ) - %
a \otimes \id_\fock ) v \evec{g} \rangle = %
\int_0^t \langle u \evec{f}, k^\cga_s( E^{\widehat{f( s )}} \cga( a ) %
E_{\widehat{g( s )}}) v \evec{g} \rangle \std s
\]
for all $u$,~$v \in \ini$, $f$,~$g \in \elltwo$, $a \in \opsp$ and
$t \in \R_+$. If $\cga$ is completely bounded then $k$ is CB regular.
\end{thm}
\begin{proof}
See \cite{LiW01}; the proof contained there is valid for any operator
space.
\end{proof}

QS cocycles are the correct limit objects for quantum random walks, as
the next section makes clear.

\subsection{Random-walk convergence}

\begin{defn}
If $\tau > 0$ and $\wga \in \kbo{\mmul}{\opsp}{\mat{\mmul}{\opsp}}$
then the \emph{embedded walk} with \emph{generator} $\wga$ and
\emph{step size} $\tau$ is the mapping process $K^{\wga, \tau}$ such
that
\[
K^{\wga, \tau}_t( a ) := (\id_\ini \otimes D_\tau)^* %
( \wga^{(n)}( a ) \otimes \id_{\bebe_{[n}} ) %
( \id_\ini \otimes D_\tau ) \qquad %
\mbox{if }t \in [n \tau, (n + 1) \tau[
\]
for all $a \in \opsp$ and $t \in \R_+$.
\end{defn}

\begin{defn}
If $\tau > 0$ and $\wga \in \kbo{\mmul}{\opsp}{\mat{\mmul}{\opsp}}$
then the modification
$\mmodf{\wga}{\tau} \in \kbo{\mmul}{\opsp}{\mat{\mmul}{\opsp}}$ is
defined by setting
\[
\mmodf{\wga}{\tau}( a ) := %
( \tau^{-1 / 2} \Delta^\perp + \Delta ) %
( \wga( a ) - a \otimes \id_\mmul ) %
( \tau^{-1 / 2} \Delta^\perp + \Delta ) \qquad \forall\, a \in \opsp,
\]
where $\Delta$ is the orthogonal projection from $\ini \otimes \mmul$
onto $\ini \otimes \mul$ and
$\Delta^\perp := \id_{\ini \otimes \mmul} - \Delta$. Note that
$\mmodf{\wga}{\tau}$ is completely bounded if $\wga$ is. In
block-matrix form,
\[
\mbox{if } \wga = \begin{bmatrix}
\wga^0_0 & \wga_\times^0\\[1ex]
\wga_0^\times & \wga_\times^\times
\end{bmatrix} \quad \mbox{then} \quad %
\mmodf{\wga}{\tau}(a) = \begin{bmatrix}
\tau^{-1}(\wga_0^0(a) - a) & \tau^{-1/2} \wga_\times^0(a) \\[1ex]
\tau^{-1/2} \wga_0^\times(a) & %
\wga_\times^\times(a) - a \otimes \id_\mul
\end{bmatrix}.
\]
\end{defn}

\begin{rem}\label{rem:conv}
For a sequence $(\wga_n)$ in $\kbo{\hilc}{\opsp}{\opsq}$, recall that
\begin{multline*}
\wga_n \matten \id_{\bop{\hilc}{}} \to 0 \mbox{ strongly} \\[0.5ex]
\mbox{if and only if }\smash[t]{\left\{
\begin{array}{ll}
\wga_n \to 0 \mbox{ strongly} & %
\mbox{when } \dim \hilc < \infty, \\[0.5ex]
\wga_n \to 0 \mbox{ in $\cb$ norm} & %
\mbox{when } \dim \hilc = \infty,
\end{array}\right.}
\end{multline*}
by \cite[Proposition~2.11 and Lemma~2.13]{Blt08}.
\end{rem}

The following is a quantum analogue of Donsker's invariance principle.

\begin{thm}\label{thm:qrw}
Let $\tau_n > 0$ and
$\wga_n$,~$\cga \in \kbo{\mmul}{\opsp}{\mat{\mmul}{\opsp}}$
be such that
\[
\tau_n \to {0+} \qquad \mbox{and} \qquad %
\lift{\mmodf{\wga_n}{\tau_n}}{\mmul} \to \lift{\cga}{\mmul}%
\mbox{ strongly}
\]
(\ie pointwise in norm) as $n \to \infty$. If $f \in \elltwo$ and
$T \in \R_+$ then
\begin{equation}\label{eqn:conv}
\lim_{n \to \infty} \sup_{t \in [0,T]} %
\| K^{\wga_n, \tau_n}_t( a ) E_{\evec{f}} - %
 k^\cga_t( a ) E_{\evec{f}} \| = 0 \qquad \forall\, a \in \opsp.
\end{equation}
If, further, $\| \mmodf{\wga_n}{\tau_n} - \cga \|_\mmulb \to 0$ as
$n \to \infty$ then
\begin{equation}\label{eqn:sconv}
\lim_{n\to\infty}\sup_{t \in [0,T]} %
\| K^{\wga_n, \tau_n}_t( \cdot ) E_{\evec{f}} - %
k^\cga_t( \cdot ) E_{\evec{f}} \|_\mmulb = 0;
\end{equation}
when $\wga_n$ and $\cga$ are completely bounded, the same implication
holds if $\| \cdot \|_\mmulb$ is replaced by~$\| \cdot \|_\cb$.
\end{thm}
\begin{proof}
See \cite[Theorem~7.6]{Blt08}.
\end{proof}

\begin{notn}
The conclusion (\ref{eqn:conv}) will be abbreviated to
$K^{\wga_n, \tau_n}\to k^\cga$; the stronger conclusion
(\ref{eqn:sconv}) will be denoted by
$K^{\wga_n, \tau_n} \tto{\mmulb} k^\cga$, or by
$K^{\wga_n, \tau_n} \tto{\cb} k^\cga$ if the completely bounded
version holds.
\end{notn}

\begin{exmp}\label{eg:hp}
Choose self-adjoint operators $\hsys \in \bop{\ini}{}$ and
$\hpar \in \bop{\mmul}{}$, and let
$V \in \bop{\ini}{\ini \otimes \mul}$. Define
$\htot(\tau) := %
\hsys \otimes \id_\mmul + \id_\ini \otimes \hpar + \hint(\tau)$
for all $\tau > 0$, where
\[
\hint(\tau) := \tau^{-1 / 2} ( Q V E^\vac + E_\vac V^* Q^* ) = %
\tau^{-1 / 2} \begin{bmatrix} 0 & V^*\\[1ex] V & 0 \end{bmatrix}
\]
and $Q : \ini \otimes \mul \hookrightarrow \ini \otimes \mmul$ is
the natural embedding map. If
$W(\tau) := \exp( -\I \tau \htot(\tau) )$ then
\begin{multline*}
( \tau^{-1/2} \Delta^\perp + \Delta ) %
( W(\tau) - I_{\ini \otimes \mmul} ) %
( \tau^{-1/2} \Delta^\perp + \Delta ) \\
 = -\I \begin{bmatrix}
\hsys + \mu \id_\ini & V^* \\[1ex] V & 0
\end{bmatrix} - %
\half \begin{bmatrix} V^* V & 0\\[1ex] 0 & 0 \end{bmatrix}
+ O(\tau)
\end{multline*}
as $\tau \to {0+}$, where
$\mu := \langle \vac, \hpar \, \vac \rangle_\mmul$. Thus if
$\tau_n \to {0+}$ and
\[
\wga(\tau) : \bop{\ini}{} \to \bop{\ini\otimes\mmul}{};\ %
a \mapsto ( a \otimes \id_\mmul ) W(\tau)
\]
then, by Theorem~\ref{thm:qrw},
$K^{\wga(\tau_n), \tau_n} \tto{\cb} k^\cga$, where
$\cga( a ) := ( a \otimes \id_\mmul ) F$ for all $a \in \bop{\ini}{}$
and
\[
F := \begin{bmatrix}
{-\I}( \hsys + \mu \id_\ini ) - \half V^* V & {-\I} V^* \\[1ex]
{-\I} V & 0
\end{bmatrix} \in \bop{\ini \otimes \mmul}{}.
\]
The cocycle $k^\cga$ obtained in the limit is an evolution of
Hudson--Parthasarathy type: if $U_t := k^\cga_t( \id_\ini )$ for all
$t \in \R_+$ then $U_t$ is unitary, by \cite[Proof of Theorem~7.1 and
Theorem~7.5]{LiW00}, and the adapted $\ini$~process $U$ satisfies
the quantum Langevin equation
\begin{equation}\label{eqn:hpqsde}
U_0 = \id_{\ini \otimes \fock}, \qquad %
\rd U_t = \rd \Lambda_F( t )U_t.
\end{equation}
(The cocycle $k^\cga$ may be recovered from $U$ by setting
$k^\cga_t( a ) := ( a \otimes \id_\fock )U_t$.)
\end{exmp}

\begin{exmp}\label{eg:eh}
Let $\tau_n$ and $W(\tau)$ be as in Example~\ref{eg:hp} and let
\[
\wga(\tau) : \bop{\ini}{} \to \bop{\ini \otimes \mmul}{} ; \ %
a \mapsto W(\tau)^* ( a \otimes \id_\mmul ) W(\tau).
\]
Then $K^{\wga(\tau_n), \tau_n} \tto{\cb} k^\cga$, by
Theorem~\ref{thm:qrw}, where
\[
\cga( a ) := \begin{bmatrix}
{-\I} [ a, \hsys ] + V^* ( a \otimes \id_\mul ) V - %
\half \{ a, V^* V \} & %
{-\I} a V^* + \I V^* ( a \otimes \id_\mul ) \\[1ex]
{-\I} ( a \otimes \id_\mul ) V + \I V a & 0
\end{bmatrix}
\]
for all $a \in \bop{\ini}{}$, with the commutator
$[ x, y ] := x y - y x$ and the anticommutator
$\{ x, y \} := x y + y x$. Here, the limit cocycle is an inner
Evans--Hudson flow: if $U$ is the unitary adapted $\ini$ process which
satisfies (\ref{eqn:hpqsde}) then, by \cite[Theorem~7.4]{LiW00},
\[
k^\cga_t( a ) = U_t^* ( a \otimes \id_\fock ) U_t \qquad %
\forall\, a \in \bop{\ini}{}, \ t \in \R_+.
\]
\end{exmp}

\begin{rem}
In Example~\ref{eg:hp}, if $\mul$ has orthonormal basis
$\{ e_j \}_{j=1}^N$, where $N$ may be infinite, then
\[
\hint(\tau) = \tau^{-1/2} \sum_{j=1}^N %
( V_j \otimes \dyad{e_j}{\vac} + V^*_j \otimes \dyad{\vac}{e_j} ),
\]
where $V_j := E^{e_j} V$ for all $j$; the series converges strongly if
$N = \infty$. For $N < \infty$, this is the dipole-interaction
Hamiltonian used by Attal and Joye in \cite{AtJ07}. With the
convention that $e_0 := \vac$, the quantum Langevin equation
(\ref{eqn:hpqsde}) takes the form
\[
U_0 = \id_{\ini \otimes \fock}, \qquad %
\rd U_t = \sum_{\alpha, \beta = 0}^N %
( F^\alpha_\beta \otimes \id_\fock ) U_t %
\std\Lambda^\beta_\alpha( t ),
\]
where $F^\alpha_\beta := E^\alpha F E_\beta$ for all $\alpha$,
$\beta$.
\end{rem}

\section{Thermal walks}\label{sec:results}

\begin{notn}
Let $\State$ be a \emph{density matrix} which acts on the Hilbert
space $\statesp$, \ie a positive operator with unit trace, and suppose
that the corresponding normal state
$\state : X \mapsto \tr( \State X )$ on the von~Neumann algebra
$\bop{\statesp}{}$ is faithful (so $\statesp$ is separable). Fix an
orthonormal basis $\{ e_j \}_{j = 0}^N$ of $\statesp$, where
$N \in \Z_+$ or $N = \infty$, such that
\[
\State = \sum_{j =0}^N \lambda_j \dyad{e_j}{e_j}.
\]
The eigenvalues $\lambda_j$ are positive and sum to~$1$, so this
series is norm convergent.
\end{notn}

\begin{defn}
Let $(\mmul, \pi, \vac)$ be the GNS representation corresponding to
$\state$ and let $X \mapsto [X]$ denote the induced mapping from
$\bop{\statesp}{}$ into $\mmul$, so that $\vac := [\id_\statesp]$.
Note that the representation
$\pi : \bop{\statesp}{} \to \bop{\mmul}{}$ is ultraweakly continuous,
injective and unital \cite[Theorems~2.3.16 and~2.4.24]{BrR87}, and
$[\ker \state]$ is dense in $\mul := \mmul \ominus \C \vac$.
\end{defn}

\begin{lem}\label{lem:pilift}
The slice map
$\sstate := \id_{\bop{\ini}{}} \uwkten \state : %
\bop{\ini \otimes \statesp}{} \to \bop{\ini}{}$
is completely positive and such that
\[
\sstate( T ) = \sum_{j = 0}^N \lambda_j E^{e_j} T E_{e_j} = %
\tr_\statesp( (\id_\ini \otimes \State) T ) %
\qquad \forall\, T \in \bop{\ini \otimes \statesp}{},
\]
where $\tr_\statesp$ is the partial trace over $\statesp$.

The unital $*$-homomorphism
$\ppi := \id_{\bop{\ini}{}} \uwkten \pi : %
\bop{\ini \otimes \statesp}{}\to\bop{\ini \otimes \mmul}{}$
is injective and such that
\begin{equation}\label{eqn:statelift}
E^{[X]} \ppi( T ) E_{[Y]} = %
\sstate( ( \id_\ini \otimes X )^* T ( \id_\ini \otimes Y ) ) \quad %
\forall\, X, Y \in \bop{\statesp}{}, \ %
T \in \bop{\ini \otimes \statesp}{}.
\end{equation}
\end{lem}
\begin{proof}
The existence of $\sstate$ and $\ppi$ as claimed
follows from \cite[Proposition~IV.5.13 and Theorem~IV.5.2]{Tak79}. The
identities are immediate if
$T \in \bop{\ini}{} \algten \bop{\statesp}{}$, so hold everywhere by
ultraweak continuity.
\end{proof}

\begin{defn}
If $\wga \in \kbo{\statesp}{\opsp}{\mat{\statesp}{\opsp}}$ then
$\ppi \circ \wga \in \kbo{\mmul}{\opsp}{\mat{\mmul}{\opsp}}$ is the
\emph{GNS generator} of the quantum random walk with \emph{generator}
$\wga$ and \emph{particle state}~$\state$. (The vector state on
$\bop{\mmul}{}$ given by $\vac$ corresponds to the state $\state$ on
$\bop{\statesp}{}$.)
\end{defn}

\begin{defn}\label{def:idem}
Fix a conditional expectation $d$ on $\bop{\statesp}{}$, \ie
a linear idempotent onto a $C^*$~subalgebra $\dialg$ such that
\begin{equation}\label{eqn:dprops}
d( X^* X ) \geq 0 \ \mbox{ and } \ %
d( d( X ) Y) = d( X ) d( Y ) = d( X d( Y ) ) \ \ %
\forall\, X, Y \in \bop{\statesp}{}.
\end{equation}
Recall that $d$ is completely positive, by
\cite[Corollary~IV.3.4]{Tak79}. Suppose further that $d$ preserves the
state~$\state$, \ie $\state \circ d = \state$. Then $d$ is ultraweakly
continuous, so~$\dialg$ is a von~Neumann algebra, and
$d( \id_\statesp ) = \id_\statesp$, since
$\id_\statesp - d( \id_\statesp)$ is an orthogonal projection with
$\state( \id_\statesp - d( \id_\statesp ) ) = 0$.
\end{defn}

\begin{notn}\label{def:diag}
Letting $\diag := \id_{\bop{\ini}{}} \uwkten d$, where $d$ is as
above, the lifted map $\diag$ is a $\sstate$-preserving conditional
expectation onto $\bop{\ini}{} \uwkten \dialg$ which leaves
$\mat{\statesp}{\opsp}$ invariant. Furthermore,
\begin{equation}\label{eqn:diag}
\diag( a \otimes \id_\statesp) = a \otimes \id_\statesp %
\quad \mbox{and} \quad %
\diag( T_1 \diag( T_2 ) ) = \diag( T_1 ) \diag( T_2 ) = %
\diag( \diag( T_1 ) T_2 )
\end{equation}
for all $a\in \bop{\ini}{}$ and
$T_1$,~$T_2 \in \bop{\ini \otimes \statesp}{}$. These, together with
the identify $\sstate \circ \diag = \sstate$, imply that
\begin{alignat}{2}
\sstate( \diag( T_1 ) T_2 ) & = \sstate( T_1 \diag( T_2 ) ) \qquad %
& & \forall\, T_1, \ T_2 \in \bop{\ini \otimes \statesp}{} %
\qquad \mbox{and} \label{eqn:slicediag}\\[0.5ex]
\sstate( ( a \otimes \id_\statesp ) T ) & = a \sstate( T ) & & %
\forall\, a \in \bop{\ini}{}, \ T \in \bop{\ini \otimes \statesp}{}.%
\label{eqn:slicemod}
\end{alignat}
\end{notn}

The conditional expectation $\diag$ plays a vital r\^{o}le in scaling
quantum random walks in order to obtain convergence. The following
example is a natural choice for this map, that given by the eigenbasis
$\{ e_j \}_{j = 0}^N$.

\begin{exmp}\label{exm:diag}
Define the \emph{diagonal map}
\[
\diag_\re : \bop{\ini \otimes \statesp}{} \to %
\bop{\ini \otimes \statesp}{}; \ T %
\mapsto ( \id_\ini \otimes S )^* %
( T \otimes \id_\statesp ) ( \id_\ini \otimes S ),
\]
where the \emph{Schur isometry}
$S \in \bop{\statesp}{\statesp \otimes \statesp}$ is such that
$S e_j = e_j \otimes e_j$ for all~$j$. If~$\dialg_\re$ is the maximal
Abelian subalgebra of $\bop{\statesp}{}$ generated by
$\{ \dyad{e_j}{e_j} \}_{j = 0}^N$ then~$\diag_\re$ is the lifting of
the unique $\state$-preserving conditional expectation $d_\re$ onto
$\dialg_\re$.
\end{exmp}

\begin{defn}
If $\tau > 0$ and
$\wga \in \kbo{\statesp}{\opsp}{\mat{\statesp}{\opsp}}$ 
then the modification
$\nmodf{\wga}{\tau} \in \kbo{\statesp}{\opsp}{\mat{\statesp}{\opsp}}$
is defined by setting
\[
\nmodf{\wga}{\tau} := \tau^{-1} \diag \circ \wga' + %
\tau^{-1 / 2} \diag^\perp \circ \wga = %
( \tau^{-1} \diag + \tau^{-1 / 2} \diag^\perp ) \circ \wga',
\]
where $\wga' \in \kbo{\statesp}{\opsp}{\mat{\statesp}{\opsp}}$ is such
that $\wga'( a ) := \wga( a ) - a \otimes \id_\statesp$ for all
$a \in \opsp$ and
$\diag^\perp := \id_{\bop{\ini \otimes \statesp}{}} - \diag$. If
$\wga$ is completely bounded then so is $\nmodf{\wga}{\tau}$.
\end{defn}

\begin{thm}\label{thm:main}
Let $\tau_n > 0$ and
$\wga_n$,~$\wgb \in \kbo{\statesp}{\opsp}{\mat{\statesp}{\opsp}}$
be such that
\[
\tau_n \to {0+} \qquad \mbox{and} \qquad %
\lift{\nmodf{\wga_n}{\tau_n}}{\statesp} \to 
\lift{\wgb}{\statesp} \mbox{ strongly}
\]
as $n \to \infty$. If
$\cgb \in \kbo{\mmul}{\opsp}{\mat{\mmul}{\opsp}}$ is defined by
setting
\begin{equation}\label{eqn:gendef}
\cgb( a ) := %
\Delta^\perp ( \ppi \circ \wgb )( a ) \Delta^\perp + %
\Delta ( \ppi \circ \diag^\perp \circ \wgb )( a ) %
\Delta^\perp + \Delta^\perp %
( \ppi \circ \diag^\perp \circ \wgb )( a ) \Delta
\end{equation}
then $\cgb$ is completely bounded whenever $\wgb$ is and
$K^{\ppi \circ \wga_n, \tau_n} \to k^\cgb$. Furthermore,
\[
\mbox{if} \quad %
\| \nmodf{\wga_n}{\tau_n} - \wgb \|_{\statesp\rb} \to 0 %
\quad \mbox{then} \quad %
K^{\ppi \circ \wga_n, \tau_n} \tto{\mmulb} k^\cgb
\]
and, when $\wga_n$ and $\wgb$ are completely bounded,
\[
\mbox{if} \quad \| \nmodf{\wga_n}{\tau_n} - \wgb \|_\cb \to 0 %
\quad \mbox{then} \quad %
K^{\ppi \circ \wga_n, \tau_n} \tto{\cb} k^\cgb.
\]
\end{thm}
\begin{proof}
Note first that
$\wga' = %
( \tau \diag + \tau^{1 / 2} \diag^\perp ) \circ \nmodf{\wga}{\tau}$
for all $\tau > 0$. If $X, \ Y \in \ker \state$ and $a \in \opsp$ then
(\ref{eqn:statelift}) and the identity $\sstate \circ \diag = \sstate$
imply that
\begin{align*}
E^\vac \mmodf{\ppi \circ \wga}{\tau}( a ) E_\vac & = %
\tau^{-1} E^\vac \ppi( \wga'( a ) ) E_\vac = %
\tau^{-1} \sstate( \wga'( a ) ) \\
& = \sstate( \nmodf{\wga}{\tau}( a ) ) = %
E^\vac ( \ppi \circ \nmodf{\wga}{\tau} )( a ) E_\vac, \\[0.5ex]
E^{[X]} \mmodf{\ppi \circ \wga}{\tau}( a ) E_\vac & = %
\tau^{-1 / 2} E^{[X]} \ppi( \wga'( a ) ) E_{\vac} \\
& = E^{[X]} ( \ppi \circ ( \tau^{1 / 2} \diag + \diag^\perp ) %
\circ \nmodf{\wga}{\tau} )( a ) E_\vac, \\[0.5ex]
E^\vac \mmodf{\ppi \circ \wga}{\tau}( a ) E_{[Y]} & = %
E^\vac ( \ppi \circ ( \tau^{1 / 2} \diag + \diag^\perp ) %
\circ \nmodf{\wga}{\tau} )( a ) E_{[Y]} \\[0.5ex]
\mbox{and} \quad %
E^{[X]} \mmodf{\ppi \circ \wga}{\tau}(a) E_{[Y]} & = %
E^{[X]} \ppi( \wga'( a ) ) E_{[Y]} \\
& = E^{[X]} ( \ppi \circ ( \tau \diag + \tau^{1 / 2} \diag^\perp ) %
\circ \nmodf{\wga}{\tau} )( a ) E_{[Y]}.
\end{align*}
Letting $\Theta := \nmodf{\wga}{\tau} - \wgb$, this working shows that
\begin{align*}
( \mmodf{\ppi \circ \wga}{\tau} - \cgb )( a ) & = \Delta^\perp %
( \ppi \circ \Theta )( a ) \Delta^\perp + %
\Delta ( \ppi \circ \diag^\perp \circ \Theta )( a ) \Delta^\perp \\
& \quad + \Delta^\perp ( \ppi \circ \diag^\perp \circ %
\Theta )( a ) \Delta + \tau^{1 / 2} R( a ),
\end{align*}
where
\begin{align*}
R( a ) & := \Delta ( \ppi \circ \diag \circ %
\nmodf{\wga}{\tau} )( a ) \Delta^\perp + \Delta^\perp ( \ppi \circ %
\diag \circ \nmodf{\wga}{\tau} )( a ) \Delta \\
& \quad + \Delta ( \ppi \circ ( \tau^{1 / 2} \diag + \diag^\perp ) %
\circ \nmodf{\wga}{\tau} )( a ) \Delta.
\end{align*}
The result follows, by Theorem~\ref{thm:qrw}.
\end{proof}

\begin{rem}\label{rem:coord}
If $\wgb \in \kbo{\statesp}{\opsp}{\mat{\statesp}{\opsp}}$ and $\cgb$
is defined by (\ref{eqn:gendef}) then the identities
(\ref{eqn:statelift}) and (\ref{eqn:slicediag}) imply that
\begin{align*}
E^\vac \cgb( a ) E_\vac & = %
\sstate( \wgb( a ) ) = E^\vac \ppi( \wgb( a ) ) E_\vac, \\[0.5ex]
E^{[X]} \cgb( a ) E_\vac & = \sstate( %
( \id_\ini \otimes X )^* ( \diag^\perp \circ \wgb )( a ) ) = %
E^{[d^\perp( X )]} \ppi( \wgb( a ) ) E_\vac, \\[0.5ex]
E^\vac \cgb( a ) E_{[Y]} & = \sstate( %
( \diag^\perp \circ \wgb )( a ) ( \id_\ini \otimes Y) ) = %
E^\vac \ppi( \wgb( a ) ) E_{[d^\perp( Y )]} \\[0.5ex]
\mbox{and} \quad E^{[X]} \cgb( a ) E_{[Y]} & = 0
\end{align*}
for all $X, \ Y \in \ker \state$ and $a \in \opsp$, where
$d^\perp := \id_{\bop{\statesp}{}} - d$. If
$n := \dim \statesp < \infty$ then at most
\[
2 \dim \{ [d^\perp( X )] : X \in \ker \state \} \leq 2 ( n^2 - 1 )
\]
independent noises appear in the quantum Langevin
equation~(\ref{eqn:coqsde}) satisfied by~$k^\cgb$, with equality in
the above if $d( X ) = \state( X ) \id_\statesp$ for all
$X \in \bop{\statesp}{}$.
\end{rem}

\begin{rem}
Although the construction in Example~\ref{exm:diag} may appear to
depend on the choice of orthonormal basis which diagonalises $\State$,
this is not really so. To see this, suppose $\{ f_j \}_{j = 0}^N$ is
another eigenbasis for $\State$, labelled so that~$e_j$ and~$f_j$ have
the same eigenvalue for all $j$, let $d_\re$ and $d_\rf$ be the
$\state$-preserving conditional expectations from $\bop{\statesp}{}$
onto the subalgebras generated by $\{ \dyad{e_j}{e_j} \}_{j = 0}^N$
and $\{ \dyad{f_j}{f_j} \}_{j = 0}^N$, and let $\diag_\re$ and
$\diag_\rf$ be their lifts to $\bop{\ini \otimes \statesp}{}$.

If $U \in \bop{\statesp}{}$ is the unique unitary operator such that
$U e_j = f_j$ for all~$j$ then
\[
\state( U^* X U ) = \state( X ) \quad \mbox{and} \quad %
U^* d_\rf( X ) U = d_\re( U^* X U ) \qquad %
\forall\, X \in \bop{\statesp}{}.
\]
Hence, if
$\wga$,~$\wgb \in \kbo{\statesp}{\opsp}{\mat{\statesp}{\opsp}}$ and 
$\widetilde{U} := \id_\ini \otimes U$,
\[
(m_{d_\re}(\wga, \tau) - \wgb)( a ) = \widetilde{U}^* %
( m_{d_\rf}(\check{\wga}, \tau) - \check{\wgb} )( a ) %
\widetilde{U} \qquad \forall\, a \in \opsp,
\]
where
$\check{\wga} : a \mapsto \widetilde{U} \wga( a ) \widetilde{U}^*$
\etc. Let $\cgb_\re$ and $\cgb_\rf$ be defined by (\ref{eqn:gendef}),
but with~$\diag$ equal to $\diag_\re$ and $\diag_\rf$, respectively,
and $\wgb$ replaced by $\check{\wgb}$ in the latter case; if
$a \in \opsp$ and $X$,~$Y \in \ker \state$ then
\begin{align*}
E^\vac \cgb_\rf( a ) E_\vac & = %
 \sstate( \check{\wgb}( a ) ) = \sstate( \wgb( a ) ) = %
E^\vac \cgb_\re( a ) E_\vac, \\[0.5ex]
E^{[X]} \cgb_\rf( a ) E_\vac & = \sstate( %
( \id_\ini \otimes X )^* %
( \diag_\rf^\perp \circ \check{\wgb} )( a ) ) \\
 & = \sstate( \widetilde{U}^* ( \id_\ini \otimes X^* ) \widetilde{U} %
\widetilde{U}^* \diag_\rf^\perp ( \check{\wgb}( a ) ) %
\widetilde{U} ) \\
 & = \sstate( ( \id_\ini \otimes U^* X U )^* %
\diag_\re^\perp ( \wgb( a ) ) ) = %
E^{[U^* X U]} \cgb_\re( a ) E_\vac \\[0.5ex]
E^\vac \cgb_\rf( a ) E_{[Y]} & = %
E^\vac \cgb_\re( a ) E_{[U^* Y U]} \\[0.5ex]
\mbox{and} \quad %
E^{[X]} \cgb_\rf( a ) E_{[Y]} & = 0 = %
E^{[U^* X U]} \cgb_\re( a ) E_{[U^* Y U]}.
\end{align*}
Thus if $W \in \bop{\mmul}{}$ is the unique unitary operator such that
$W[X]=[U^* X U]$ for all $X\in\bop{\statesp}{}$ then
\[
\cgb_\rf( a ) = (\id_\ini \otimes W)^* \cgb_\re( a ) %
(\id_\ini \otimes W ) \qquad \forall\, a \in \opsp;
\]
the change of orthonormal basis used to define the diagonal map
is manifest as a change of coordinates (an isometric isomorphism
of~$\mmul$ which preserves $\vac$) and unitary conjugation of the
map~$\wgb$.
\end{rem}

\section{Examples}\label{sec:egs}

Henceforth $\opsp$ will be a von~Neumann algebra $\vna$ and
$\mat{\statesp}{\vna} = \vna \uwkten \bop{\statesp}{}$. As above, $d$
is a conditional expectation on $\bop{\statesp}{}$ which preserves the
faithful normal state $\state$ and $\diag = \id_{\bop{\ini}{}} \uwkten
d$.

\subsection{Hudson--Parthasarathy evolutions}

\begin{rem}
Suppose $F \in \vna \uwkten \bop{\statesp}{}$ and let
$\wgb : \vna \to \vna \uwkten \bop{\statesp}{}$ be such that
$\wgb( a ) = ( a \otimes \id_\statesp ) F$ for all $a \in \vna$. If
$\cgb$ is defined by~(\ref{eqn:gendef}) then
$\cgb( a ) = ( a \otimes \id_\mmul ) G$ for all $a \in \vna$, by
(\ref{eqn:diag}), where
\begin{equation}\label{eqn:gdef}
G = \Delta^\perp \ppi( F )
\Delta^\perp + \Delta \ppi( \diag^\perp( F ) ) \Delta^\perp + %
\Delta^\perp \ppi( \diag^\perp( F ) ) \Delta.
\end{equation}
Furthermore, the cocycle $k^\cgb$ is such that
$k^\cgb_t( a ) = ( a \otimes \id_\fock ) X_t$ for all $a \in \vna$ and
$t \in \R_+$, where the adapted $\ini$~process
$X = \{ X_t := k^\cgb_t( \id_\ini ) \}_{t \in \R_+}$ satisfies the
Hudson--Parthasarathy equation
\begin{equation}\label{eqn:hpee}
X_0 = \id_{\ini \otimes \fock}, \qquad %
\rd X_t = \rd \Lambda_G( t ) X_t.
\end{equation}
(See \cite[Proof of Theorem~7.1]{LiW00}.)
\end{rem}

\begin{thm}\label{thm:warmhp}
Let $\hdiag$ and $\hoffd$ be self-adjoint elements of
$\vna \uwkten \bop{\statesp}{}$ such that $\hdiag = \diag( \hdiag )$
and $\hoffd = \diag^\perp( \hoffd )$. If $\tau > 0$ and
$\htot(\tau) := \hdiag + \tau^{-1 / 2} \hoffd$ then the completely
isometric map
\begin{equation}\label{eqn:warmhpgen}
\wga(\tau) : \vna \to \vna \uwkten \bop{\statesp}{}; \ %
a \mapsto ( a \otimes \id_\statesp ) \exp( {-\I} \tau \htot(\tau) )
\end{equation}
is such that $\| \nmodf{\wga(\tau)}{\tau} - \wgb \|_\cb \to 0$ as
$\tau \to {0+}$, where
\[
\wgb : \vna \to \vna \uwkten \bop{\statesp}{}; \ %
a \mapsto ( a \otimes \id_\statesp ) %
( {-\I} ( \hdiag + \hoffd ) - \half \diag( \hoffd ^2 ) ).
\]
Thus $K^{\ppi \circ \wga(\tau_n), \tau_n} \tto{\cb} k^\cgb$ if
$\tau_n \to {0+}$, where
$\cgb : \vna \to \vna \uwkten \bopp(\mmul)$ is such that
\begin{align*}
E^\vac \cgb( a ) E_\vac & = {-\I} a \sstate( \hdiag ) - %
\half a \sstate( \hoffd^2 ), \\[0.5ex]
E^{[X]} \cgb( a ) E_\vac & = {-\I} a \sstate( %
( \id_\ini \otimes X )^* \hoffd ), \\[0.5ex]
E^\vac \cgb( a ) E_{[Y]} & = {-\I} a \sstate( %
\hoffd ( \id_\ini \otimes Y ) ) \\[0.5ex]
\mbox{and} \quad E^{[X]} \cgb( a ) E_{[Y]} & = 0
\end{align*}
for all $X$,~$Y \in \ker \state$ and $a \in \vna$, and
$U_t := k^\cgb_t( \id_\ini)$ is unitary for all $t \in \R_+$.
\end{thm}
\begin{proof}
If $a \in \vna$ then
$\diag ( \wga(\tau)'( a ) ) = ( a \otimes \id_\statesp ) %
\diag( \exp( {-\I} \tau \htot(\tau) ) - %
\id_{\ini \otimes \statesp} )$,
by (\ref{eqn:diag}), and the same holds with $\diag$ replaced by
$\diag^\perp$. As $\tau \to {0+}$,
\begin{align*}
\tau^{-1} \diag( \exp( {-\I} \tau \htot(\tau) ) - %
\id_{\ini \otimes \statesp} ) & = {-\I} \hdiag - %
\half \diag( \hoffd ^2 ) + O( \tau^{1 / 2} ) \\[0.5ex]
\mbox{and} \quad \tau^{-1 / 2} %
\diag^\perp( \exp( {-\I} \tau \htot(\tau) ) - %
\id_{\ini \otimes \statesp} ) & = {-\I} \hoffd + O( \tau^{1 / 2} ),
\end{align*}
which gives the first claim. Theorem~\ref{thm:main} and
Remark~\ref{rem:coord}, simplified using the identities
(\ref{eqn:diag}), (\ref{eqn:slicemod}) and
$\sstate = \sstate \circ \diag$, complete the result; unitarity
holds for the adapted $\ini$~process $U = \{ U_t \}_{t \in \R_+}$ by
\cite[Theorem~7.5]{LiW00}.
\end{proof}

\begin{exmp}\label{eg:ajhp}
Let $\diag = \diag_\re$ be the diagonal map of
Example~\ref{exm:diag} and suppose
$\hdiag := \hsys \otimes \id_\statesp + \id_\ini \otimes \hpar$, where
the self-adjoint operators $\hsys \in \vna$ and
\[
\hpar = \sum_{j = 0}^N \mu_j \dyad{e_j}{e_j} \in \bop{\statesp}{},
\]
with this series strongly convergent when $N=\infty$. Let
$\statesp_\times := \statesp \ominus \C e_0$, choose
$V \in \vna \uwkten \bop{\C}{\statesp_\times}$ and define
\[
\hoffd := Q V E^{e_0} + E_{e_0} V^* Q^* = %
\begin{bmatrix} 0 & V^* \\[1ex] V & 0 \end{bmatrix},
\]
where
$Q : %
\ini \otimes \statesp_\times \hookrightarrow \ini \otimes \statesp$
is the natural embedding. If $\tau_n \to {0+}$ and $\wga(\tau)$ is
defined by (\ref{eqn:warmhpgen}) then Theorem~\ref{thm:warmhp} implies
that $K^{\ppi \circ \wga(\tau_n), \tau_n} \tto{\cb} k^\cgb$, where
\begin{align*}
E^\vac \cgb( a ) E_\vac & = %
{-\I} a ( \hsys + \state( \hpar ) \id_\ini ) - %
\half a \sstate( \hoffd^2 ), \\[0.5ex]
E^{[X]} \cgb(a) E_\vac & = {-\I} a ( E^{X \State e_0} Q V + %
V^* Q^* E_{\State X^* e_0} ), \\[0.5ex]
E^\vac \cgb( a ) E_{[Y]} & = {-\I} a ( E^{\State Y^* e_0} Q V + %
V^* Q^* E_{Y \State e_0} ) \\[0.5ex]
\mbox{and} \quad E^{[X]} \cgb( a ) E_{[Y]} & = 0 & 
\end{align*}
for all $X$,~$Y \in \ker \state$ and $a \in \vna$. If
$\dim \statesp < \infty$ then this agrees with
\cite[Theorem~7]{AtJ07}. Letting $\lambda_0 = 1$ and
$\lambda_j = 0$ for all $j > 0$, the above is also in formal agreement
with Example~\ref{eg:hp}; note that the GNS space is spanned by
$\{ [\dyad{e_j}{e_0}] \}_{j = 0}^N$ in this case.
\end{exmp}

\subsection{Evans--Hudson evolutions}

\begin{defn}
An \emph{Evans--Hudson flow} is a solution~$k$ of the quantum
stochastic differential equation (\ref{eqn:coqsde}) which is
$*$-homomorphic, \ie each $k_t( a )$ extends to a bounded operator on
$\ini \otimes \fock$ and the mapping $a \mapsto k_t( a )$ is a
$*$-homomorphism from~$\vna$ to $\bop{\ini \otimes \fock}{}$ for all
$t \in \R_+$.
\end{defn}

\begin{rem}
If $F \in \vna \uwkten \bop{\statesp}{}$ and
$\wgb : \vna \to \vna \uwkten \bop{\statesp}{}$ such that
\[
\wgb( a ) = %
( a \otimes \id_\statesp ) F + F^* ( a \otimes \id_\statesp ) + %
\diag( \diag^\perp( F )^* ( a \otimes \id_\statesp ) %
\diag^\perp( F ) ) \qquad \forall\, a \in \vna
\]
then a short calculation shows that $\cgb$ defined by
(\ref{eqn:gendef}) is such that
\[
\cgb( a ) = %
( a \otimes \id_\mmul ) G + G^* ( a \otimes \id_\mmul ) + %
G^* \Delta ( a \otimes \id_\mmul ) \Delta G \qquad %
\forall\, a \in \vna,
\]
with $G$ as in (\ref{eqn:gdef}). It follows \cite[Theorem~7.4]{LiW00}
that $k^\cgb_t( a ) = X_t^* ( a \otimes \id_\mmul ) X_t$ for all
$a \in \vna$ and $t \in \R_+$, where $X$ is the solution of
(\ref{eqn:hpee}); if $X$ is co-isometric then~$k^\cgb$ is an
\emph{inner} Evans--Hudson flow.
\end{rem}

Recall that $[ x, y ] := x y - y x$ and $\{ x, y \} := x y + y x$ are
the commutator and the anticommutator, respectively.

\begin{thm}\label{thm:warmeh}
If $\hdiag$, $\hoffd$ and $\htot$ are as in Theorem~\ref{thm:warmhp}
then the ultraweakly continuous unital $*$-homomorphism
\begin{equation}\label{eqn:eh}
\wga(\tau) : \vna \to \vna \uwkten \bop{\statesp}{}; \ %
a \mapsto \exp( \I \tau \htot ) ( a \otimes \id_\statesp ) %
\exp( {-\I} \tau \htot )
\end{equation}
is such that
$\| \nmodf{\wga(\tau)}{\tau} - \wgb \|_\cb \to 0$ as $\tau \to {0+}$,
where
\begin{align*}
\wgb : \vna & \to \vna \uwkten \bop{\statesp}{}; \\
a & \mapsto {-\I} [ a \otimes \id_\statesp, \hdiag + \hoffd ] + %
\diag( \hoffd ( a \otimes \id_\statesp ) \hoffd ) - %
\half \{ a \otimes \id_\statesp, \diag( \hoffd^2 ) \}.
\end{align*}
Hence $K^{\ppi \circ \wga(\tau_n), \tau_n} \tto{\cb} k^\cgb$ if
$\tau_n \to {0+}$, where $\cgb : \vna \to \vna \uwkten \bopp(\mmul)$
is completely bounded and
\begin{align*}
E^\vac \cgb( a ) E_\vac & = {-\I} [ a, \sstate( \hdiag ) ] %
+ \sstate( \hoffd ( a \otimes \id_\statesp ) \hoffd ) %
- \half \{ a, \sstate( \hoffd ^2 ) \}, \\[0.5ex]
E^{[X]} \cgb( a ) E_\vac & = {-\I} [ a, \sstate( %
( \id_\ini \otimes X )^* \hoffd ) ], \\[0.5ex]
E^\vac \cgb( a ) E_{[Y]} & = {-\I} [ a, \sstate( %
\hoffd ( \id_\ini \otimes Y ) ) ] \\[0.5ex]
\mbox{and} \quad E^{[X]} \cgb( a ) E_{[Y]} & = 0
\end{align*}
for all $X$,~$Y \in \ker \state$ and $a \in \vna$, and $k^\cgb$ is an
inner Evans--Hudson flow.
\end{thm}
\begin{proof}
This follows in the same manner as Theorem~\ref{thm:warmhp}.
\end{proof}

\begin{rem}
Since $\wga(\tau)$ in (\ref{eqn:eh}) is an ultraweakly continuous
$*$-homomorphism, so are $\wga(\tau)^{(n)}$, for all $n \in \Z_+$, and
the embedded walk~$K^{\ppi \circ \wga_n(\tau_n), \tau_n}$. It follows,
by strong convergence, that the limit cocycle $k^\cgb$ of
Theorem~\ref{thm:warmeh} is $*$-homomorphic; this shows directly that
$k^\cgb$ is an Evans--Hudson flow.
\end{rem}

\begin{exmp}\label{eg:ajeh}
If $\diag$, $\hdiag$, $\hoffd$ and $Q$ are as in Example~\ref{eg:ajhp}
and the generator~$\wga(\tau)$ is defined by (\ref{eqn:eh}) then
$K^{\ppi \circ \wga(\tau_n), \tau_n} \tto{\cb} k^\cgb$ if
$\tau_n \to {0+}$, where $\cgb$ is such that
\begin{align*}
E^\vac \cgb( a ) E_\vac & = {-\I} [ a, \hsys ] + %
\sstate( \hoffd (a \otimes \id_\statesp ) \hoffd ) - %
\half \{ a, \sstate( \hoffd^2 ) \}, \\[0.5ex]
E^{[X]} \cgb(a) E_\vac & = {-\I} [ a, %
E^{X \State e_0 } Q V + V^* Q^* E_{\State X^* e_0} ], \\[0.5ex]
E^\vac \cgb( a ) E_{[Y]} & = {-\I} [ a, %
E^{\State Y^* e_0} Q V + V^* Q^* E_{Y \State e_0} ] \\[0.5ex]
\mbox{and} \quad E^{[X]} \cgb( a ) E_{[Y]} & = 0
\end{align*}
for all $X$,~$Y \in \ker \state$ and $a \in \vna$. When
$\dim \statesp < \infty$, the map $a\mapsto E^\vac \cgb( a ) E_\vac$
is the Lindblad generator of \cite[Corollary~13]{AtJ07}. Formally, the
above agrees with Example~\ref{eg:eh} when $\lambda_0 = 1$ and
$\lambda_j = 0$ for all $j > 0$.
\end{exmp}

\phantomsection%
\addcontentsline{toc}{section}{Acknowledgements}%
\section*{Acknowledgements}
The author completed part of this work while an Embark Postdoctoral
Fellow at University College Cork, funded by the Irish Research
Council for Science, Engineering and Technology. It was begun at the
International Workshop on Quantum Probability and its Applications
held in Ferrazzano, near Campobasso, Italy, thanks to the hospitality
of the meeting's organiser, Professor Michael Skeide. Earlier drafts
of this work received generous comments from Professor Martin Lindsay
and Dr Adam Skalski; an observation of Dr Stephen Wills led to the
significant improvement of Section~3.

\phantomsection%
\addcontentsline{toc}{section}{References}%

\end{document}